\documentclass[11pt]{article}

\usepackage{amsmath,amsthm,amssymb}
\usepackage{pstricks,pst-grad,pst-node,pst-tree}
\usepackage{setspace}
\usepackage{mathptmx}
\usepackage{authblk}
\usepackage{graphicx}
\usepackage{algorithm2e}
\usepackage{subfigure}
\usepackage{longtable}
\usepackage{appendix}

\theoremstyle{plain}
\newtheorem{theorem}{Theorem}

\newtheorem{observation}[theorem]{Observation}

\newcommand{\seq}[1]{\ensuremath{\left( #1 \right)}}
\newcommand{\ceil}[1]{\ensuremath{\left\lceil #1 \right\rceil}}
\newcommand{\floor}[1]{\ensuremath{\left\lfloor #1 \right\rfloor}}
\newcommand{\fr}[1]{\ensuremath{\left\lbrace #1 \right\rbrace}}

\newcommand{\reg}{\operatorname{rg}}

\newcommand{\flfr}[2]{\ensuremath{\left\lfloor \frac{#1}{#2} \right\rfloor}}
\newcommand{\frfr}[2]{\ensuremath{\left\lbrace \frac{#1}{#2} \right\rbrace}}

\newcommand{\df}[4]{\ensuremath{D\left(#1,\ #2,\ #3,\ #4\right)}}
\newcommand{\rf}[4]{\ensuremath{(#1-#2)\Big(\frac{#4-#1-1}{#1 #4-#3}+\frac{#2}{#3-#2 #4}\Big)}}

\newcounter{casenum}
\newenvironment{caseof}{\setcounter{casenum}{1}}{\vskip.5\baselineskip}
\newcommand{\case}[2]{\vskip.5\baselineskip\par\noindent {\bfseries Case \arabic{casenum}:} #1\\#2\addtocounter{casenum}{1}}

\title{How Much Regularity Forces a Sequence to be Graphic?
\footnote{Official contribution of the National Institute of Standards and Technology; not subject to copyright in the United States.}
}

\author{Brian~Cloteaux\\
National Institute of Standards and Technology,\\
Applied and Computational Mathematics Division,\\
Gaithersburg, MD}

\date {}

\begin{document}

\maketitle

\begin{abstract}
For an integer sequence (with even sum), the closer that the sequence is to being regular, the more likely that the sequence is graphic. But how regular must a sequence be before it must always be graphic?  We show that for many sequences if all values are within $\frac{n-2}{4}$ of the mean degree value, then the sequence is graphic.  We also see how this result extends to show when a maximum difference between sequence values implies that a sequence is graphic.
\end{abstract}

\section{Introduction}

We will assume that a degree sequence is a sequence of non-negative integers whose sum is even. In others words, it is a sequence that could potentially be graphic, i.e. be a list of the number of adjacent edges for the nodes for some simple graph. 
When we speak a how regular a sequence is, we are speaking of how close, in some sense, the values in a sequence are to each other. Thus, a completely regular sequence will have all the same values.

 The definition of regularity in this article will be the maximum distance any of the values are in the sequence from the mean value of the sequence.  Another possible way to define regularity is to look at the difference between the largest and smallest value in the sequence.  We will also look at this definition by showing an application of our main result to this second measure.

One reason that we are interested in measuring the regularity of a sequence is because 
there is a close relationship between how regular a sequence is and whether that sequence is graphic.  In general, the more regular a sequence is, the more likely that the sequence is graphic.  This culminates with the regular or near-regular degree sequences (ones where the difference between the largest and smallest values is no more than one) always being graphic (\cite{Chen:1988}, Lemma 1).

We can think about this relationship of regularity and being graphic through the majorization (or dominance) operator. Without going into depth about the operator itself, one sequence majorizes another if the second one is ``more regular'' than the first in some sense. This idea is extended with the 
result that says if a graphic sequence majorizes another sequence, then the second sequence must also be graphic (\cite{Ruch:1979}, Theorem 1).  Thus, the majorization operator forms a lattice  over the integer sequences \cite{Brylawski:1973} where the graphic sequences are clustered at the bottom of the lattice.  So, for any non-graphic degree sequence, if we follow a chain down the lattice starting from that sequence, we will eventually reach a graphic sequence.  In other words, as we make a sequence more regular, it will eventually become graphic.

The question arises then, how regular must a sequence be in order to guarantee that it is graphic? It is obvious that simply being near-regular is a weak bound to this question. In this article, we give a tight bound to this question based on the measure of the  maximum difference from the mean value for the values in a sequence.  Additionally, we show that this result has implications between the relationship of the size of the maximum difference in a sequence and whether that sequence is forced to be graphic.

\section{Definitions and Needed Results}
We begin our discussion with  some needed definitions and results.  A {\it degree sequence} $\pi=(d_1, ..., d_n)$ is a sequence of $n$ number of non-negative integers. If a simple graph exists whose number of adjacent edges for each node matches the sequence $\pi$, then we say that $\pi$ is {\it graphic}.  When speaking about specific sequences, we will use a superscript to denote a repeated value in the sequence as a shorthand notation, e.g., $\seq{3^2} = \seq{3,3}$.  The largest value in a sequence $\pi$ is denoted as $\Delta(\pi)$, while the smallest is denoted as $\delta(\pi)$.  The sum of the sequence $\pi$ is
\begin{equation}
	s = \sum_{i=1}^{n} d_i,
\end{equation}
and so its mean degree value is $\frac{s}{n}$.
We use the standard notation of $\floor{x}$, $\ceil{x}$, and $\lbrace x \rbrace$ to denote the floor, ceiling, and fractional part of the value of $x$.

The {\it complement} of a degree sequence $\pi$
is the sequence 
\begin{equation}
\bar{\pi} = \left( n-1-d_n, ..., n-d_1-1 \right).
\end{equation}
It is straightforward to see that $\pi$ is graphic if and only if $\bar{\pi}$ is graphic.  

We say that a sequence $\pi$ is $c$-regular if every value $d_i \in \pi$ is within a distance of $c$ from the mean degree value, i.e. $\frac{s}{n}-c \leq d_i \leq \frac{s}{n}+c$.  We denote the smallest value of $c$ for a sequence as 
\begin{equation}
\reg(\pi) = \min \lbrace c\in \mathbb{R}^{+} |\  \forall d_i \in \pi,\  \frac{s}{n}-c \leq d_i \leq \frac{s}{n}+c \rbrace.
\end{equation}

An extremely useful tool in our discussion will be the following function,
\begin{equation} \label{eqn:d_define}
\df{x_1}{x_2}{x_3}{x_4} = \rf{x_1}{x_2}{x_3}{x_4}.
\end{equation}
The reason for considering this function follows from this theorem.
\begin{theorem}[\cite{Cloteaux:2018}, Theorem 3]
\label{thm:edge_bound}
Let $\pi$ be an integer sequence of length $n$ such that
$n-1 \geq \Delta(\pi) \geq \delta(\pi) \geq 0$  and with even
sum $s$ where $n\Delta(\pi) > s > n\delta(\pi)$.
If
\begin{equation} \label{eqn:edge_bound}
\df{\Delta(\pi)}{\delta(\pi)}{s}{n} = \rf{\Delta(\pi)}{\delta(\pi)}{s}{n} \geq 1,
\end{equation} 
then $\pi$ is graphic.
\end{theorem}
Note that Equation \eqref{eqn:d_define} is
invariant for complement sequences, meaning that 
\begin{equation}
\df{a}{b}{s}{n} = \df{\bar{a}}{\bar{b}}{\bar{s}}{n},
\end{equation}
where $\bar{a} = (n-1)-b$, $\bar{b} = (n-1)-a$, and  $\bar{s} = n(n-1) - s$.

By substituting the values $\frac{s}{n}+c$ and $\frac{s}{n}-c$ for the top and bottom range of the values in a sequence, we derive the following equation,
\begin{equation} \label{eqn:d_val}
   \df{\frac{s}{n}+c}{\frac{s}{n}-c}{s}{n} = \frac{2(n-(2c+1))}{n}.
\end{equation}
From Equation \eqref{eqn:d_val}, we make several observations.
\begin{observation}
For the values $s_1,s_2,n,c_1,c_2,d \in \mathbb{R}^{+}_{*}$, then
\begin{itemize}
\item[]
\begin{equation} \label{eqn:d_equality}
\df{\frac{s_1}{n}+\frac{n-2}{4}}{\frac{s_1}{n}-\frac{n-2}{4}}{s_1}{n} = 1,
\end{equation}
%\item
%\begin{equation} \label{eqn:d_slide}
%\df{\frac{s_1}{n}+c_1}{\frac{s_1}{n}-c_1}{s_1}{n} =
%\df{\frac{s_2}{n}+c_1}{\frac{s_2}{n}-c_1}{s_2}{n},
%\end{equation}
\item[]
\begin{equation} \label{eqn:d_shrink}
\df{\frac{s_1}{n}+c_1}{\frac{s_1}{n}-c_1}{s_1}{n} =
\df{\frac{s_1}{n}+c_2}{\frac{s_1}{n}-c_2}{s_1}{n} +
\frac{4(c_2-c_1)}{n},
\end{equation}
\item[]
\begin{equation} \label{eqn:d_asym}
\df{\frac{s_1}{n}+c_1}{\frac{s_1}{n}-c_2}{s_1}{n} =
\df{\frac{s_2}{n}+c_1}{\frac{s_2}{n}-c_2}{s_2}{n}+
\frac{d}{n} \left( \frac{c_2^2-c_1^2}{c_1 c_2} \right),
\end{equation}
where $s_2 = s_1 + nd$.
\end{itemize}
\label{obs:d_func}
\end{observation}

In addition, if we rewrite the sum as $s = n \left( \flfr{s}{n} + \frfr{s}{n} \right)$, then we can derive the form,  
\begin{equation} \label{eqn:floor_form}
    \df{\flfr{s}{n}+c}{\flfr{s}{n}-c}{s}{n} = 2c \left( \frac{(cn - c -2c^2) + \frfr{s}{n}(n-2\flfr{s}{n}-1)}{n(c+\frfr{s}{n})(c-\frfr{s}{n})} \right).
\end{equation}
It follows from Equation \eqref{eqn:floor_form}, that for $s_1$ and $s_2$ where $\flfr{s_1}{n} = \flfr{s_2}{n}$ and $s_1 \leq s_2$ then,
\begin{equation} \label{eqn:fl_cmp}
\df{\flfr{s_1}{n}+c}{\flfr{s_1}{n}-c}{s_1}{n} \leq \df{\flfr{s_2}{n}+c}{\flfr{s_2}{n}-c}{s_2}{n}
\end{equation}
\section{Values Around the Degree Mean}

We are now ready to state our main result. This result breaks the behavior of the sequences into cases that depend on the value of their degree mean.  One point that we will revisit is that for the sequences whose mean degree is in the center of possible values, the following bound only depends on the length of the sequence.

\begin{theorem}
Let $\pi$ be an integer sequence of length $n$ such that $0 \leq \delta(\pi) \leq
\Delta(\pi) \leq n-1$ and with even sum $s$. The
sequence $\pi$ is graphic if,
\begin{enumerate}
\item $\reg(\pi) \leq \frac{n-2}{4}$ where $\frac{n-2}{4} \leq \frac{s}{n} \leq \frac{3n-2}{4}$,
\item $\reg(\pi) \leq n-1 - \frac{s}{n}$ where $\frac{3n-2}{4} < \frac{s}{n} \leq n-1$,
\item $\reg(\pi) \leq \frac{s}{n}$ where  $0 \leq \frac{s}{n} < \frac{n-2}{4}$.
\end{enumerate}
\end{theorem}

\begin{proof}
In order to show Part 1 of this result, we want to show that for any sequence $\pi$ with even sum, where $\frac{n-2}{4} \leq \frac{s}{n} \leq \frac{3n-2}{4}$ and $\Delta(\pi) = \floor{ \frac{s}{n} + \frac{n-2}{4}}$ and $\delta(\pi) =  \ceil{ \frac{s}{n} - \frac{n-2}{4}}$ is graphic. We notice that there are four possible cases for these values depending on the fractional parts of $\frac{s}{n}$ and $\frac{n-2}{4}$.  These cases are the combinations of the following equations:
\begin{equation} \label{eqn:up_bnd}
   \floor{ \frac{s}{n} + \frac{n-2}{4}} =
   \begin{cases}
        \floor{\frac{s}{n}}+\floor{\frac{n-2}{4}}+1,& 
        \text{if } \fr{\frac{s}{n}}+\fr{\frac{n-2}{4}} \geq 1 \\
        \floor{\frac{s}{n}}+\floor{\frac{n-2}{4}},&         \text{otherwise,}
    \end{cases}
\end{equation}

\begin{equation} \label{eqn:lw_bnd}
   \ceil{ \frac{s}{n} - \frac{n-2}{4}} =
   \begin{cases}
        \floor{\frac{s}{n}}-\floor{\frac{n-2}{4}}+1,& 
        \text{if } \fr{\frac{s}{n}} > \fr{\frac{n-2}{4}} \\
        \floor{\frac{s}{n}}-\floor{\frac{n-2}{4}},&         \text{otherwise.}
    \end{cases}
\end{equation}

We will find working with the complement sequences useful in certain situations.  We can link the bounds between a sequence and its complement with the following observation,
\begin{equation} \label{eqn:comp_frac}
\frfr{s}{n} = 1 - \frfr{\bar{s}}{n}.
\end{equation}
Extending the Equation \eqref{eqn:comp_frac} to the conditions for  Equations \eqref{eqn:up_bnd} and \eqref{eqn:lw_bnd} we derive a connection between these bounds to their complement sequences,
\begin{equation}
 \fr{\frac{s}{n}} \leq \fr{\frac{n-2}{4}} \Longleftrightarrow 1 \leq \fr{\frac{\bar{s}}{n}}+\fr{\frac{n-2}{4}}.
\end{equation}

We now show this result by examining it in term of Equation \eqref{eqn:d_define}. For each case, we want to show that this value is greater than or equal to one.
\begin{caseof}
\case{$\df{\floor{\frac{s}{n}+\frac{n-2}{4}}}{\ceil{\frac{s}{n}- \frac{n-2}{4}}}{s}{n} = \df{\flfr{s}{n}+\flfr{n-2}{4}}{\flfr{s}{n}- \flfr{n-2}{4}}{s}{n} $}

This case follows from an application of our earlier observations.
{\footnotesize
  \setlength{\abovedisplayskip}{6pt}
  \setlength{\belowdisplayskip}{\abovedisplayskip}
  \setlength{\abovedisplayshortskip}{0pt}
  \setlength{\belowdisplayshortskip}{3pt}
\begin{multline}
 \df{\flfr{s}{n}+\flfr{n-2}{4}}{\flfr{s}{n}- \flfr{n-2}{4}}{s}{n} \\
 \begin{aligned}
 &\geq  \df{\flfr{s}{n}+\flfr{n-2}{4}}{\flfr{s}{n}- \flfr{n-2}{4}}{s-n\frfr{s}{n}}{n} & \text{(Eq. \ref{eqn:fl_cmp})}\\
 &\geq \df{\flfr{s}{n}+\frac{n-2}{4}}{\flfr{s}{n}- \frac{n-2}{4}}{s-n\frfr{s}{n}}{n} & \text{(Eq. \ref{eqn:d_asym})} \\
 &= 1 & \text{(Eq. \ref{eqn:d_equality})} \\
\end{aligned}
\end{multline}
}%

\case{$\df{\floor{\frac{s}{n}+\frac{n-2}{4}}}{\ceil{\frac{s}{n}- \frac{n-2}{4}}}{s}{n} = \df{\flfr{s}{n}+\flfr{n-2}{4}+1}{\flfr{s}{n}- \flfr{n-2}{4}+1}{s}{n} $}

If $\frfr{s}{n} > \frfr{n-2}{4}$, that implies $1 > \frfr{n-2}{4} + \frfr{\bar{s}}{n}$. Also, if $\frfr{s}{n} + \frfr{n-2}{n} \geq 1$ implies that $\frfr{\bar{s}}{n} \leq \frfr{n-2}{4}$. Thus this case reduces to Case 1 through its complement sequence:
{\footnotesize
  \setlength{\abovedisplayskip}{6pt}
  \setlength{\belowdisplayskip}{\abovedisplayskip}
  \setlength{\abovedisplayshortskip}{0pt}
  \setlength{\belowdisplayshortskip}{3pt}
\begin{multline}
\df{\flfr{s}{n}+\flfr{n-2}{4}+1}{\flfr{s}{n}- \flfr{n-2}{4}+1}{s}{n} \\
\begin{aligned}
  &= \df{\flfr{\bar{s}}{n}+\flfr{n-2}{4}}{\flfr{\bar{s}}{n}- \flfr{n-2}{4}}{\bar{s}}{n} \\
 &\geq 1.
\end{aligned}
\end{multline}
}%
\case{$\df{\floor{\frac{s}{n}+\frac{n-2}{4}}}{\ceil{\frac{s}{n}- \frac{n-2}{4}}}{s}{n} = \df{\flfr{s}{n}+\flfr{n-2}{4}+1}{\flfr{s}{n}- \flfr{n-2}{4}}{s}{n} $}

For this case we make two observations.  The first is, from Equation \eqref{eqn:comp_frac},  we can assume without a loss of generality that $\frfr{s}{n} \leq \frac{1}{2}$ (or else we could simply work with the complement sequence).  This means that $n\frfr{s}{n} \leq \frac{n}{2}$.

Additionally, since  from the case conditions $\frfr{n-2}{4} \geq \frfr{s}{n}$ and $\frfr{s}{n} + \frfr{n-2}{4} \geq 1$, which implies that $\frfr{n-2}{4} \geq \frfr{\bar{s}}{n}$, then we establish that  $\frfr{n-2}{4} \geq \frac{1}{2}$ or, equivalently,  $\flfr{n-2}{4} + \frac{1}{2} \leq \frac{n-2}{4}$.   Using these observations,  we show the following sequence: 
{\footnotesize
  \setlength{\abovedisplayskip}{6pt}
  \setlength{\belowdisplayskip}{\abovedisplayskip}
  \setlength{\abovedisplayshortskip}{0pt}
  \setlength{\belowdisplayshortskip}{3pt}
\begin{multline}
\df{\flfr{s}{n}+\flfr{n-2}{4}+1}{\flfr{s}{n}-\flfr{n-2}{4}}{s}{n}\\
\begin{aligned}
 &= \df{\left(\flfr{s}{n}+\frac{1}{2}\right)+\left(\flfr{n-2}{4}+\frac{1}{2}\right)}{\left(\flfr{s}{n}+\frac{1}{2}\right)-\left(\flfr{n-2}{4}+\frac{1}{2}\right)}{s}{n} \\
  &\geq \df{\left(\flfr{s}{n}+\frac{1}{2}\right)+\left(\flfr{n-2}{4}+\frac{1}{2}\right)}{\left(\flfr{s}{n}+\frac{1}{2}\right)-\left(\flfr{n-2}{4}+\frac{1}{2}\right)}{s-n\frfr{s}{n}+\frac{n}{2}}{n}\\
 &\geq \df{\left(\flfr{s}{n}+\frac{1}{2}\right)+\frac{n-2}{4}}{\left(\flfr{s}{n}+\frac{1}{2}\right)-\frac{n-2}{4}}{s-n\frfr{s}{n}+\frac{n}{2}}{n}\\
 &= 1. \\
\end{aligned}
\end{multline}
}%

\case{$\df{\floor{\frac{s}{n}+\frac{n-2}{4}}}{\ceil{\frac{s}{n}- \frac{n-2}{4}}}{s}{n} = \df{\flfr{s}{n}+\flfr{n-2}{4}}{\flfr{s}{n}- \flfr{n-2}{4}+1}{s}{n} $}

To show $\df{\flfr{s}{n}+\flfr{n-2}{4}}{\flfr{s}{n}- \flfr{n-2}{4}+1}{s}{n} \geq 1$, we begin with the equivalent expression 
{\footnotesize
  \setlength{\abovedisplayskip}{6pt}
  \setlength{\belowdisplayskip}{\abovedisplayskip}
  \setlength{\abovedisplayshortskip}{0pt}
  \setlength{\belowdisplayshortskip}{3pt}
\begin{multline}
\df{\flfr{s}{n}+\flfr{n-2}{4}}{\flfr{s}{n}- \flfr{n-2}{4}+1}{s}{n} = \\ \df{\left(\flfr{s}{n}+\frac{1}{2}\right)+\left(\flfr{n-2}{4}-\frac{1}{2}\right)}{\left(\flfr{s}{n}+\frac{1}{2}\right)-\left(\flfr{n-2}{4}-\frac{1}{2}\right)}{s}{n},\\
\end{multline}
}%
and then argue similarly to the last case.
\end{caseof}
This establishes that $\df{\floor{\frac{s}{n}+\frac{n-2}{4}}}{\ceil{\frac{s}{n}- \frac{n-2}{4}}}{s}{n} \geq 1$, and then from Theorem \ref{thm:edge_bound} that the sequence $\pi$ is graphic.
For the other two cases, we can immediately use identical reasoning as the first case, but with the added constraints that $\Delta(\pi) \leq n-1$ and $\delta(\pi) \geq 0$.
\end{proof}

In general, this result is a tight bound on how regular a sequence must be in order to force it to be graphic.  In order to see this, we define a family of sequences where $n$ evenly divides $s$ (i.e., $n|s$) and where $n$ is even.  We write out the sequences in this family as
\begin{equation}
\seq{ \left(\frac{s}{n}+c \right)^{\frac{n}{2}}, \left( \frac{s}{n}-c \right)^{\frac{n}{2}} }.
\end{equation}
If these sequences are graphic then the following Erd\H{o}s-Gallai inequality \cite{Erdos:1960} must hold: 
\begin{equation}
\frac{n}{2} \left( \frac{s}{n}+c \right) \leq \frac{n}{2} \left( \frac{n}{2} - 1 \right) + \frac{n}{2} \left( \frac{s}{n} -c \right). 
\end{equation}
It is easy to see that this inequality holds only when 
\begin{equation}
c \leq \frac{n-2}{4},
\end{equation}
thus this family of sequences is always non-graphic when $c$ is greater than our bound.

\section{Graphic Differences}
We now come back to the point that for sequences where $\frac{n-2}{4} \leq \frac{s}{n} \leq \frac{3n-2}{4}$, the regularity bound only depends on the length of the sequence. As a consequence, we can extend this result to apply to the maximum difference ($\Delta(\pi) - \delta(\pi)$) of a sequence $\pi$ in this set.

Since our earlier bound is only dependent on the length of those sequence, we can see that there is a function $m(n)$ such that if
\begin{equation}
    \Delta(\pi) - \delta(\pi) \leq m(n),
\end{equation}
then $\pi$ must be graphic.  In fact, the previous result provides bounds for the function $m(n)$.  The value of $m(n)$ cannot be greater than $\frac{n-2}{2}$, since we have already seen a counterexample in the previous section.  Additionally, if we set the difference $\Delta(\pi) - \delta(\pi) = \frac{n-2}{4}$, then these sequences will also be graphic from the last theorem, no matter the value of the mean degree.  Thus a simple bound for $m(n)$ is
\begin{equation}
\frac{n-2}{4} \leq m(n) \leq \frac{n-2}{2}.
\end{equation}

While we do not have a precise formulation of the function $m(n)$, we have performed a computational investigation of it.  
In Appendix \ref{sec:table_diff}, we present a list of maximum graphic distances for the lengths up to 100 composed from an exhaustive computer search.  From an examination of these values, it appears that the lower bound on $m(n)$ is close to $\frac{5}{12}n$, which is much larger than $\frac{n-2}{4}$.  Showing an exact bound or formulation for $m(n)$ remains an open research problem.

{ \singlespacing
\bibliographystyle{nist}
\bibliography{regularity_bound}
}

\appendix
\section{Table of Graphic Differences} \label{sec:table_diff}
This table shows the maximum values $m(n)$ such that if a degree sequence $\pi$ where $\frac{n-2}{4} \leq \frac{s}{n} \leq \frac{3n-2}{4}$, and $\Delta(\pi) - \delta(\pi) \leq m(n)$, then $\pi$ is graphic.
%\begin{center}
%\begin{tabular}{ |c|c|c| } 
 %\hline
 \begin{center}
 \begin{longtable}{|c|c|c|c|}
\caption{Maximum Graphic Differences}\\
\hline
\textbf{$n$} & \textbf{$m(n)$} & \textbf{Minimal Non-graphic Example} \\
\hline
\endfirsthead
\multicolumn{3}{c}%
{\tablename\ \thetable\ -- \textit{Continued from previous page}} \\
\hline
\textbf{$n$} & \textbf{$m(n)$} & \textbf{Minimal Non-graphic Example} \\
\hline
\endhead
\hline \multicolumn{4}{r}{\textit{Continued on next page}} \\
\endfoot
\hline
\endlastfoot
4 & 1 & $(1^{2}, 3^{2})$ \\
5 & 1 & $(2^{2}, 4^{3})$ \\
6 & 2 & $(1^{4}, 4^{2})$ \\
7 & 2 & $(1^{4}, 4^{3})$ \\
8 & 3 & $(1^{6}, 5^{2})$ \\
9 & 3 & $(1^{6}, 4^{1}, 5^{2})$ \\
10 & 3 & $(1^{7}, 5^{3})$ \\
11 & 4 & $(1^{8}, 6^{3})$ \\
12 & 4 & $(1^{8}, 4^{1}, 6^{3})$ \\
13 & 4 & $(1^{8}, 6^{5})$ \\
14 & 5 & $(1^{9}, 5^{1}, 7^{4})$ \\
15 & 5 & $(1^{9}, 6^{1}, 7^{5})$ \\
16 & 6 & $(1^{10}, 6^{1}, 8^{5})$ \\
17 & 6 & $(1^{10}, 6^{1}, 8^{6})$ \\
18 & 7 & $(1^{11}, 7^{1}, 9^{6})$ \\
19 & 7 & $(1^{11}, 8^{1}, 9^{7})$ \\
20 & 7 & $(2^{13}, 10^{7})$ \\
21 & 8 & $(1^{12}, 8^{1}, 10^{8})$ \\
22 & 8 & $(2^{14}, 5^{1}, 11^{7})$ \\
23 & 9 & $(1^{13}, 10^{1}, 11^{9})$ \\
24 & 9 & $(2^{15}, 6^{1}, 12^{8})$ \\
25 & 10 & $(1^{14}, 10^{1}, 12^{10})$ \\
26 & 10 & $(2^{16}, 7^{1}, 13^{9})$ \\
27 & 11 & $(1^{15}, 12^{1}, 13^{11})$ \\
28 & 11 & $(2^{17}, 8^{1}, 14^{10})$ \\
29 & 11 & $(2^{17}, 8^{1}, 14^{11})$ \\
30 & 12 & $(2^{18}, 9^{1}, 15^{11})$ \\
31 & 12 & $(2^{18}, 10^{1}, 15^{12})$ \\
32 & 12 & $(3^{20}, 4^{1}, 16^{11})$ \\
33 & 13 & $(2^{19}, 10^{1}, 16^{13})$ \\
34 & 13 & $(3^{21}, 5^{1}, 17^{12})$ \\
35 & 14 & $(2^{20}, 12^{1}, 17^{14})$ \\
36 & 14 & $(3^{22}, 6^{1}, 18^{13})$ \\
37 & 15 & $(2^{21}, 12^{1}, 18^{15})$ \\
38 & 15 & $(3^{23}, 7^{1}, 19^{14})$ \\
39 & 16 & $(2^{22}, 14^{1}, 19^{16})$ \\
40 & 16 & $(3^{24}, 8^{1}, 20^{15})$ \\
41 & 16 & $(4^{27}, 19^{1}, 21^{13})$ \\
42 & 17 & $(3^{25}, 9^{1}, 21^{16})$ \\
43 & 17 & $(4^{28}, 22^{15})$ \\
44 & 18 & $(3^{26}, 10^{1}, 22^{17})$ \\
45 & 18 & $(3^{26}, 10^{1}, 22^{18})$ \\
46 & 18 & $(4^{29}, 22^{1}, 23^{16})$ \\
47 & 19 & $(3^{27}, 12^{1}, 23^{19})$ \\
48 & 19 & $(4^{30}, 24^{18})$ \\
49 & 20 & $(3^{28}, 12^{1}, 24^{20})$ \\
50 & 20 & $(4^{30}, 5^{1}, 25^{19})$ \\
51 & 21 & $(3^{29}, 14^{1}, 25^{21})$ \\
52 & 21 & $(4^{31}, 6^{1}, 26^{20})$ \\
53 & 22 & $(3^{30}, 14^{1}, 26^{22})$ \\
54 & 22 & $(4^{32}, 7^{1}, 27^{21})$ \\
55 & 22 & $(5^{35}, 25^{1}, 28^{19})$ \\
56 & 23 & $(4^{33}, 8^{1}, 28^{22})$ \\
57 & 23 & $(5^{36}, 24^{1}, 29^{20})$ \\
58 & 24 & $(4^{34}, 9^{1}, 29^{23})$ \\
59 & 24 & $(5^{37}, 27^{1}, 30^{21})$ \\
60 & 25 & $(4^{35}, 10^{1}, 30^{24})$ \\
61 & 25 & $(5^{38}, 28^{1}, 31^{22})$ \\
62 & 25 & $(5^{38}, 31^{24})$ \\
63 & 26 & $(5^{39}, 31^{1}, 32^{23})$ \\
64 & 26 & $(5^{39}, 29^{1}, 32^{24})$ \\
65 & 27 & $(4^{37}, 12^{1}, 32^{27})$ \\
66 & 27 & $(5^{40}, 31^{1}, 33^{25})$ \\
67 & 28 & $(4^{38}, 14^{1}, 33^{28})$ \\
68 & 28 & $(5^{41}, 33^{1}, 34^{26})$ \\
69 & 28 & $(6^{43}, 35^{26})$ \\
70 & 29 & $(5^{42}, 35^{28})$ \\
71 & 29 & $(6^{44}, 32^{1}, 36^{26})$ \\
72 & 30 & $(5^{42}, 6^{1}, 36^{29})$ \\
73 & 30 & $(7^{47}, 17^{1}, 38^{25})$ \\
74 & 31 & $(5^{43}, 7^{1}, 37^{30})$ \\
75 & 31 & $(6^{46}, 30^{1}, 38^{28})$ \\
76 & 32 & $(5^{44}, 8^{1}, 38^{31})$ \\
77 & 32 & $(6^{47}, 31^{1}, 39^{29})$ \\
78 & 32 & $(7^{49}, 19^{1}, 40^{28})$ \\
79 & 33 & $(6^{48}, 34^{1}, 40^{30})$ \\
80 & 33 & $(7^{50}, 21^{1}, 41^{29})$ \\
81 & 34 & $(6^{49}, 35^{1}, 41^{31})$ \\
82 & 34 & $(7^{51}, 23^{1}, 42^{30})$ \\
83 & 35 & $(6^{50}, 38^{1}, 42^{32})$ \\
84 & 35 & $(7^{52}, 25^{1}, 43^{31})$ \\
85 & 35 & $(8^{54}, 12^{1}, 44^{30})$ \\
86 & 36 & $(6^{51}, 38^{1}, 43^{34})$ \\
87 & 36 & $(8^{55}, 15^{1}, 45^{31})$ \\
88 & 37 & $(6^{52}, 40^{1}, 44^{35})$ \\
89 & 37 & $(7^{54}, 28^{1}, 45^{34})$ \\
90 & 38 & $(6^{53}, 42^{1}, 45^{36})$ \\
91 & 38 & $(7^{55}, 31^{1}, 46^{35})$ \\
92 & 38 & $(8^{57}, 16^{1}, 47^{34})$ \\
93 & 39 & $(7^{56}, 32^{1}, 47^{36})$ \\
94 & 39 & $(8^{58}, 18^{1}, 48^{35})$ \\
95 & 40 & $(7^{57}, 35^{1}, 48^{37})$ \\
96 & 40 & $(8^{59}, 20^{1}, 49^{36})$ \\
97 & 40 & $(9^{62}, 50^{35})$ \\
98 & 41 & $(8^{60}, 22^{1}, 50^{37})$ \\
99 & 41 & $(9^{63}, 50^{1}, 51^{35})$ \\
100 & 42 & $(8^{61}, 24^{1}, 51^{38})$ \\ 
% \hline
%\end{tabular}
%\end{center}
\end{longtable}
\end{center}

%\section{Scratch space}

%\begin{theorem} 
%\label{thm:sum_incr}
%For $nb < s < s+\epsilon < na$ and $c < \frac{s(n-1)}{n^2}$ then
%\begin{equation}
%    D(\frac{s}{n}+c,\frac{s}{n}-c,s,n) < %D(\frac{s+\epsilon}{n}+c,\frac{s+\epsilon}{n}-c,s+\epsilon,n).
%\end{equation}

%\end{theorem}

%\begin{proof}
%Taking the derivative of $D(\frac{s}{n}+c,\frac{s}{n}-c,s,n)$ %with respect to $s$,  we derive
%\begin{align}
%\frac{\partial D(\frac{s}{n}+c,\frac{s}{n}-c,s,n)}{\partial s} & = 2c \left( \frac{1}{n^2c-s+ns} - \frac{1}{n^2c+s-ns} \right) \\
%   & = \frac{2c(1-n)}{(n^2c-s+ns)(n^2c+s-ns)}
%\end{align}
%Since $c < \frac{s(n-1)}{n^2}$, then $n^2c+s-ns < 0$.  Including the facts that $2c(1-n)<0$ and $n^2c-s+ns > 0$, then $\frac{\partial D(a,b,s,n)}{\partial s} > 0$ implying that $D(\frac{s}{n}+c,\frac{s}{n}-c,s,n) < D(\frac{s}{n}+c,\frac{s}{n}+c,s+\epsilon,n)$.
%\end{proof}

%This bound implies an upper bound on the largest length that can contain
%a non-graphic difference.  
%\begin{equation}
%	n \geq 8d-6,
%\end{equation}
%Not a tight upper bound, as for the $d=3$ then the longest non-graphic
%sequence, found through exhaustive search, is $(5,5,5,5,2,2,2)$ for a length of 7.  The formula puts the upper bound as 18. 

%{\bf About the authors:} Brian Cloteaux is a computer scientist with the
%Applied and Computational Mathematics Division of the Information
%Technology Laboratory. 

\end{document}